\documentclass[11pt]{amsart}

\usepackage{amsmath,amssymb,amsthm,mathtools}
\usepackage{enumitem}
\usepackage{aliascnt}
\usepackage[hidelinks]{hyperref}
\usepackage[nameinlink,capitalise]{cleveref}
\usepackage{setspace}

\newtheorem{theorem}{Theorem}[section]
\newaliascnt{proposition}{theorem}
\newtheorem{proposition}[proposition]{Proposition}
\aliascntresetthe{proposition}
\newaliascnt{lemma}{theorem}
\newtheorem{lemma}[lemma]{Lemma}
\aliascntresetthe{lemma}
\newaliascnt{corollary}{theorem}
\newtheorem{corollary}[corollary]{Corollary}
\aliascntresetthe{corollary}
\theoremstyle{definition}
\newaliascnt{definition}{theorem}
\newtheorem{definition}[definition]{Definition}
\aliascntresetthe{definition}
\newaliascnt{example}{theorem}
\newtheorem{example}[example]{Example}
\aliascntresetthe{example}
\theoremstyle{remark}
\newaliascnt{remark}{theorem}
\newtheorem{remark}[remark]{Remark}
\aliascntresetthe{remark}

\newcommand{\LLY}{\mathrm{LLY}}
\newcommand{\DN}{\mathrm{DN}}
\newcommand{\Etl}{E_{\mathrm{tl}}}
\newcommand{\core}{\operatorname{core}_2}
\newcommand{\Dist}{\operatorname{dist}}
\newcommand{\Skel}{\mathcal S}

\title[Curvature-Distortion Numbers of Graphs]{Curvature-Distortion Numbers of Graphs: Prescribed Lin--Lu--Yau Curvature Lower Bounds}
\author{Qing Xia}
\address{School of Mathematical Sciences\\
University of Science and Technology of China\\
96 Jinzhai Road\\
Hefei 230026, Anhui Province\\
China}
\email{xq0420@mail.ustc.edu.cn}
\date{\today}


\begin{document}

\begin{abstract}
We study an optimization problem for weighted Lin--Lu--Yau curvature.
Given a graph $G=(V,E)$ and a prescribed lower curvature bound $k\ge0$,
we ask for the smallest multiplicative distortion
\[
\Dist(w):=\frac{\sup_{e\in E}w_e}{\inf_{e\in E}w_e}
\]
of the edge weights needed to make every edge have curvature at least
$k$. This leads to the family
\[
\DN_{\LLY}^{k}(G)
:=
\inf\bigl\{
\Dist(w):
\kappa_{\LLY}^{w}(e)\ge k
\text{ for every }e\in E(G)
\bigr\},
\]
which extends the nonnegative-curvature invariant introduced in our
preceding paper \cite{Xia2026InvariantI}. The edge weights change the
transition probabilities, while the combinatorial graph distance
remains fixed.

Positive prescribed curvature leads to a fundamental distinction from
the case $k=0$. If $G$ is a connected locally finite infinite graph and
$k>0$, then
\[
\DN_{\LLY}^{k}(G)=\infty.
\]
Thus no finite multiplicative distortion of the edge weights can impose
a uniform positive curvature lower bound on such a graph.

For a finite tree $T$ with $m$ edges, such weights exist exactly
when $k\le2/m$. We prove that the minimum distortion is attained
by a weight unique up to scaling, and that the normalized
minimizer and its distortion depend continuously and
nondecreasingly on $k$.

We then extend the finite-tree theory to general connected graphs. As an
application, we characterize the graphs of girth at least six for which
finite $k$-distortion is possible.
\end{abstract}

\maketitle

\section{Introduction and statement of results}
\label{sec:intro}

Our setting combines the transport curvature introduced by Ollivier
\cite{Ollivier2009} with the Lin--Lu--Yau high-idleness limit
\cite{LinLuYau2011}; the idleness formulation is further developed in
\cite{BCLMP2018}. We use normalized transition probabilities coming from
positive edge weights, in the general weighted-Laplacian framework of
M\"unch and Wojciechowski \cite{MunchWojciechowski2019}, but keep the
underlying combinatorial distance as the transportation cost. This is also
the convention adopted by Lin and Liu in their prescribed-curvature flow
\cite{LinLiu2026}.

In our preceding paper \cite{Xia2026InvariantI}, we introduced the
Lin--Lu--Yau curvature-distortion number and developed the case in which
all edge curvatures are required to be nonnegative. The present paper
replaces the lower bound $0$ by a prescribed number $k\ge0$ and asks for
the least multiplicative distortion needed to realize that bound.

Throughout, graphs are simple and undirected. For a graph $G=(V,E)$ we
write $x\sim y$, or equivalently $xy\in E$, when $x$ and $y$ are
adjacent, put $N_G(x):=\{y\in V:y\sim x\}$, and write
$d_x:=|N_G(x)|$ for the combinatorial degree. A graph is
\emph{locally finite} if $d_x<\infty$ for every vertex $x$. We define
curvature-distortion numbers only for connected locally finite graphs.
For a positive edge weight $w:E\to(0,\infty)$, write
$w_{xy}=w_e$ when $e=xy$, and define the weighted degree at $x$ by
\[
d_x^w:=\sum_{y\sim x}w_{xy}.
\]
The weighted Lin--Lu--Yau curvature is denoted by
$\kappa_{\LLY}^w$ and is recalled precisely in
\cref{subsec:weighted-LLY}. Let
\[
\kappa_{\LLY}^w(G):=\inf_{e\in E}\kappa_{\LLY}^w(e).
\]
Then $\kappa_{\LLY}^w(G)\ge k$ means that $\kappa_{\LLY}^w(e)\ge k$ for every $e\in E(G)$; in
particular, $\kappa_{\LLY}^w(G)>0$ implies strict positivity on every
edge.

For a positive edge weight $w$, define the \emph{distortion} of $w$ by
\[
\Dist(w):=\frac{\sup_{e\in E}w_e}{\inf_{e\in E}w_e},
\]
and take $\Dist(w)=\infty$ whenever the quotient is not finite. For finite $G$, this is equivalent to
$\Dist(w)=\frac{\max_{e\in E}w_e}{\min_{e\in E}w_e}$. We then
define the \emph{$k$-LLY-distortion number} of $G$ by
\begin{equation}
\label{eq:DNk-intro}
\DN_{\LLY}^{k}(G)
:=\inf\left\{
\Dist(w):
\kappa_{\LLY}^{w}(G)\ge k
\right\}.
\end{equation}
A positive weight satisfying $\kappa_{\LLY}^{w}(G)\ge k$ is called
\emph{$k$-admissible}; it is called \emph{$k$-optimal} if its distortion
equals $\DN_{\LLY}^{k}(G)$. A weight has \emph{finite distortion}
if $\Dist(w)<\infty$. When $E\ne\varnothing$, a finite-distortion
weight is called \emph{normalized} if
$\inf_{e\in E}w_e=1$; for finite $E$ this is equivalent to
$\min_{e\in E}w_e=1$. If no finite-distortion $k$-admissible weight
exists, we set $\DN_{\LLY}^{k}(G)=\infty$, and by convention
$\DN_{\LLY}^{k}(K_1)=1$. Scale invariance makes
\eqref{eq:DNk-intro} a scale-invariant variational problem. For $k=0$
we recover the invariant of \cite{Xia2026InvariantI}, and we write
\[
\DN_{\LLY}(G):=\DN_{\LLY}^{0}(G).
\]

An edge is \emph{tree-like} if it lies on no cycle of length $3$, $4$,
or $5$. We write $\Etl(G)$ for the set of tree-like edges and
\[
\Skel(G):=(V(G),\Etl(G))
\]
for the \emph{tree-like skeleton}; a component of $\Skel(G)$ is
\emph{nontrivial} if it contains an edge. For a tree $T$, a vertex of
degree at least $3$ is a \emph{branch vertex}, and
\[
B(T):=\{v\in V(T):d_v\ge3\}.
\]

The local curvature inequality on a tree-like edge is explicit. If
$xy\in\Etl(G)$ and
\[
A_x=d_x^w-w_{xy},\qquad A_y=d_y^w-w_{xy},
\]
then, for $0\le k<2$,
\begin{equation}
\label{eq:k-local-intro}
\kappa_{\LLY}^w(x,y)\ge k
\quad\Longleftrightarrow\quad
w_{xy}\ge\Psi_k(A_x,A_y),
\end{equation}
where
\[
\Psi_k:[0,\infty)^2\longrightarrow[0,\infty)
\]
is given by
\begin{equation}
\label{eq:Psi-intro}
\Psi_k(a,b)
:=
\frac{k(a+b)+\sqrt{k^2(a-b)^2+16ab}}{2(2-k)}.
\end{equation}
At $k=0$, $\Psi_0(a,b)=\sqrt{ab}$. For a finite tree $T$ and
$x\in[1,\infty)^{E(T)}$, put
\[
A_{u\mid v}(x)
:=\sum_{\substack{f\in E(T)\\ f\ni u,\ f\ne uv}}x_f
\]
and define
\begin{equation}
\label{eq:Fk-intro}
(F_{T,k}(x))_{uv}
:=\max\left\{1,
\Psi_k(A_{u\mid v}(x),A_{v\mid u}(x))
\right\}.
\end{equation}
For $k=0$ this is exactly the all-edge LLY-distortion map used in
\cite{Xia2026InvariantI}. Thus $F_{T,k}$ extends the fixed-point system
from the nonnegative-curvature case to a prescribed lower bound.

The relevant definitions for the following theorem are given in \cref{def:k-restoration-map}.

\begin{theorem}[Finite-tree $k$-profile]
\label{thm:tree-profile-intro}
Let $T$ be a finite tree with $m:=|E(T)|\ge1$. Then
\begin{equation}
\label{eq:tree-feasible-range-intro}
\DN_{\LLY}^{k}(T)<\infty
\quad\Longleftrightarrow\quad
0\le k\le\frac{2}{m}.
\end{equation}
If $m=1$, then $T=K_2$ and
\[
\DN_{\LLY}^{k}(K_2)=1
\qquad(0\le k\le2).
\]
Assume $m\ge2$. For every $k\in[0,2/m]$, the map $F_{T,k}$ has a
unique normalized fixed point
$w_T^{\star,k}\in[1,\infty)^{E(T)}$, and
\[
\DN_{\LLY}^{k}(T)
=\max_{e\in E(T)}w_{T,e}^{\star,k}.
\]
This is the unique normalized $k$-optimal weight and satisfies
\[
(w_{T,e}^{\star,k}-1)
\bigl(\kappa_{\LLY}^{w_T^{\star,k}}(e)-k\bigr)=0
\qquad(e\in E(T)).
\]
For every edge $e$, both
\[
k\longmapsto w_{T,e}^{\star,k}
\quad\text{and}\quad
k\longmapsto\DN_{\LLY}^{k}(T)
\]
are nondecreasing and continuous on $[0,2/m]$. At $k=2/m$,
\[
\kappa_{\LLY}^{w_T^{\star,2/m}}(e)=\frac{2}{m}
\qquad(e\in E(T)).
\]
\end{theorem}

At $k=0$ this recovers the finite-tree canonical fixed point of
\cite{Xia2026InvariantI}. At $k=2/m$ it reaches the
constant-curvature regime of Lin and Liu
\cite[Theorem~3.1]{LinLiu2026}. For intermediate values of $k$, the theorem determines the optimal
weights and their distortion between these two endpoints.

For $k>0$, a connected locally finite infinite graph cannot carry a
$k$-admissible weight. This follows from the standard discrete
Bonnet--Myers argument in the fixed-distance weighted model.

\begin{proposition}[Discrete Bonnet--Myers]
\label{prop:BM}
Let $G$ be connected and locally finite. If a positive edge weight $w$
satisfies
\[
\kappa_{\LLY}^{w}(e)\ge k>0
\qquad(e\in E(G)),
\]
then
\begin{equation}
\label{eq:BM}
\operatorname{diam}G\le\frac{2}{k}.
\end{equation}
\end{proposition}

Consequently, every connected locally finite infinite graph satisfies
$\DN_{\LLY}^{k}(G)=\infty$ for $k>0$, see \cref{cor:no-infinite-positive-k}.

The second main ingredient is an exact response formula for hanging
trees. Suppose a finite tree $C$ with $t$ vertices is attached to an
exterior vertex $x$ through one boundary edge $e=xr$. Writing
$p_x^w(e)=w_e/d_x^w$, every $k$-admissible weighting must satisfy
\begin{equation}
\label{eq:root-response-intro}
p_x^w(e)\ge\frac{kt}{2},
\end{equation}
and this condition is also sufficient for filling the hanging tree
with positive weights whose internal curvatures equal $k$. Equivalently,
if $A_x=d_x^w-w_e$ and $kt<2$, then
\[
w_e\ge\frac{kt}{2-kt}A_x.
\]
The response depends only on $t$, not on the shape of the hanging tree.

This gives an exact positive-curvature version of the $k=0$ $2$-core
reduction in \cite{Xia2026InvariantI}. For a finite graph $G$, let
$\core(G)$ denote its $2$-core, obtained by repeatedly deleting vertices
of degree $0$ or $1$; see \cref{def:2-core}. Suppose $K:=\core(G)\ne\varnothing$. Every
component of $G-V(K)$ is then a tree attached to a unique vertex of $K$
by one edge. Set
\[
A_G:=\{x\in V(K):N_G(x)\setminus V(K)\ne\varnothing\},
\]
let $h_G(x)$ be the total number of vertices in the components of
$G-V(K)$ attached at $x$, and form $\widehat K_G$ by adjoining, for each $x\in A_G$, a new
pendant vertex $x^\partial$ adjacent only to $x$. We call
$xx^\partial$ the \emph{proxy edge} at $x$. A
positive weight $\widehat w$ on $\widehat K_G$ is called
\emph{$k$-loaded admissible} if
\[
\kappa_{\LLY}^{\widehat w}(e)\ge k
\ \ (e\in E(K)),\ \text{and}\ \kappa_{\LLY}^{\widehat w}(xx^\partial)\ge k h_G(x)
\ \ (x\in A_G);
\]
see exact definitions in Section~\ref{subsec:loaded-core}.

\begin{theorem}[$k$-loaded $2$-core reduction]
\label{thm:loaded-core-intro}
Let $G$ be a finite connected graph with nonempty $2$-core, and let
$k\ge0$. With the notation above,
\begin{equation}
\label{eq:k-loaded-core-equivalence}
\DN_{\LLY}^{k}(G)<\infty
\quad\Longleftrightarrow\quad
\widehat K_G\text{ admits a $k$-loaded admissible weight}.
\end{equation}
\end{theorem}

At $k=0$ the proxy-edge load condition is automatic, so this reduces to
the $2$-core reduction of \cite{Xia2026InvariantI}. For $k>0$,
$h_G(x)$ is the exact curvature cost of the entire hanging forest.

We finally quantify how positive lower curvature restricts the tree-like
skeleton defined above.

\begin{theorem}[Quantitative tree-like skeleton]
\label{thm:quantitative-skeleton-intro}
Let $G$ be connected and locally finite, let $0<k<2$, and put
\[
D:=\DN_{\LLY}^{k}(G)<\infty.
\]
Then $\Skel(G)$ is a forest. Every nontrivial component $T$ satisfies
\begin{align}
\DN_{\LLY}^{k}(T)&\le D,\label{eq:skeleton-DNk}\\
|E(T)|&\le\left\lfloor\frac{2}{k}\right\rfloor,\label{eq:skeleton-size}\\
|B(T)|&\le\lceil D\rceil.\label{eq:skeleton-branch}
\end{align}
\end{theorem}

\begin{corollary}[Girth at least six]
\label{cor:girth6-intro}
Let $G$ be a finite connected graph with $E(G)\ne\varnothing$ and girth
at least $6$, and let $k>0$. Then
\begin{equation}
\label{eq:girth6-classification-intro}
\DN_{\LLY}^{k}(G)<\infty
\quad\Longleftrightarrow\quad
G\text{ is a tree and }k\le\frac{2}{|E(G)|}.
\end{equation}
\end{corollary}
Together with the finite and infinite high-girth classifications at
$k=0$ in \cite{Xia2026InvariantI}, this completes the locally finite
high-girth picture for all $k\ge0$.

\bigskip

\paragraph{Prescribed-flow interpretation.}
For a finite tree $T$ with $m$ edges, $w_T^{\star,2/m}$ is the normalized
constant-curvature weight. Lin and Liu \cite[Theorem~3.2]{LinLiu2026}
show that the prescribed flow with target $2/m$ converges exponentially
to its scaling class. Thus
\[
\DN_{\LLY}^{2/m}(T)
=
\Dist\bigl(w_T^{\star,2/m}\bigr),
\]
while for $0\le k<2/m$, $\DN_{\LLY}^{k}(T)$ gives the least distortion
needed to make every edge curvature at least $k$. Related discrete
Einstein metrics on trees are studied in
\cite{BaiChengHua2026,Cheng2026EinsteinTrees} under a different
weighted-distance convention.

\bigskip

\paragraph{Organization.}
Section~\ref{sec:prelim} fixes notation, recalls weighted LLY curvature,
and records the specific results used from Lin--Liu and
\cite{Xia2026InvariantI}. Section~\ref{sec:positive-core} proves
Bonnet--Myers (\cref{prop:BM}) and the loaded $2$-core reduction (\cref{thm:loaded-core-intro}).
Section~\ref{sec:finite-trees} develops the finite-tree profile (\cref{thm:tree-profile-intro}), and
Section~\ref{sec:skeleton} treats tree-like skeletons (\cref{thm:quantitative-skeleton-intro}) and high-girth
graphs (\cref{cor:girth6-intro}).

\section{Preliminaries}
\label{sec:prelim}

Graphs are simple and undirected throughout. The invariant
$\DN_{\LLY}^k$ is considered only for connected graphs; if such a graph
is infinite, local finiteness is always assumed. Except where explicitly
stated otherwise, the graphs below are finite.

For $G=(V,E)$, vertices of degree $1$ are called \emph{leaves},
and edges incident with leaves are \emph{pendant}. The combinatorial
distance between $x,y\in V$ is denoted by $d_G(x,y)$, or by $d(x,y)$ when no ambiguity can
arise. We adopt the convention that forests have infinite girth.

\subsection{Weighted Lin--Lu--Yau curvature}
\label{subsec:weighted-LLY}

Fix a positive weight $w:E\to(0,\infty)$. For $\alpha\in[0,1]$, let
\begin{equation}
\label{eq:mu}
\mu_x^{\alpha,w}(z)=
\begin{cases}
\alpha, & z=x,\\[2mm]
(1-\alpha)\dfrac{w_{xz}}{d_x^w}, & z\sim x,\\[3mm]
0, & \text{otherwise}.
\end{cases}
\end{equation}
If $\mu$ and $\nu$ are finitely supported probability measures on
$V$, denote their couplings by $\Pi(\mu,\nu)$. With the combinatorial
distance as cost, set
\[
W_1(\mu,\nu)
:=\inf_{\pi\in\Pi(\mu,\nu)}
\sum_{u,v\in V}d(u,v)\pi(u,v).
\]
For an edge $x\sim y$, define
\begin{equation}\label{eq:kappa_alpha_W}
    \kappa_\alpha^w(x,y)
:=1-W_1(\mu_x^{\alpha,w},\mu_y^{\alpha,w}),
\end{equation}
and
\begin{equation}\label{eq:kappa_LLY}
    \kappa_{\LLY}^w(x,y)
:=\lim_{\alpha\to1^-}\frac{\kappa_\alpha^w(x,y)}{1-\alpha}.
\end{equation}
The definitions above use the standard idleness convention of
\cite{LinLuYau2011,BCLMP2018} and the normalized weighted-Laplacian
setting of M\"unch and Wojciechowski \cite{MunchWojciechowski2019}.
For the unit weight $\mathbf1$ we suppress the superscript $w$, and we
write $\kappa_{\LLY}^{w,G}$ when the ambient graph needs to be shown.

Because the transition probabilities depend only on ratios of edge
weights, multiplying all weights by the same positive constant leaves the
curvature unchanged:
\begin{equation}
\label{eq:scale}
\kappa_{\LLY}^{cw}(e)=\kappa_{\LLY}^{w}(e)
\qquad(c>0).
\end{equation}
For a finite graph $G$, the map
\[
w\longmapsto
\bigl(\kappa_{\LLY}^w(e)\bigr)_{e\in E(G)}
\]
is locally Lipschitz on $(0,\infty)^{E(G)}$ 
\cite[Lemma~2.2]{LinLiu2026}.

For tree-like edges we use the local identity of
M\"unch and Wojciechowski
\cite[Example~2.3]{MunchWojciechowski2019}; see also
\cite[Eq.~(8)]{LinLiu2026}.

\begin{lemma}[Tree-like edge formula]
\label{lem:tree-like-formula}
Let $G$ be connected and locally finite, let $w$ be a positive edge
weight, and let $xy\in\Etl(G)$. Then
\begin{equation}
\label{eq:tree-like-curvature}
\kappa_{\LLY}^w(x,y)
=2w_{xy}\left(\frac1{d_x^w}+\frac1{d_y^w}\right)-2.
\end{equation}
\end{lemma}

For an incidence of $e$ with $x$, put
\begin{equation}
\label{eq:edge-share}
p_x^w(e):=\frac{w_e}{d_x^w}.
\end{equation}
Thus, on a tree-like edge $e=xy$,
\begin{equation}
\label{eq:share-curvature}
\kappa_{\LLY}^w(e)\ge k
\quad\Longleftrightarrow\quad
p_x^w(e)+p_y^w(e)\ge1+\frac{k}{2}.
\end{equation}

\subsection{The \texorpdfstring{$k$}{k}-LLY-distortion number}

For $k\ge0$, $\DN_{\LLY}^{k}(G)$ is defined by
\eqref{eq:DNk-intro} for connected locally finite graphs only. No
curvature-distortion number is assigned to a disconnected graph in this
paper. Recall that the value is $\infty$ if no finite-distortion $k$-admissible
weight exists, and a
$k$-admissible weight is \emph{$k$-optimal} if its distortion equals
$\DN_{\LLY}^{k}(G)$. When $E(G)\ne\varnothing$, \eqref{eq:scale}
allows every finite-distortion weight to be normalized by
$\inf_{e\in E} w_e=1$; for finite $G$, this is equivalent to
$\min_{e\in E} w_e=1$.

The attainment argument is the positive-$k$ analogue of the finite-graph
compactness argument in \cite[Proposition~2.5]{Xia2026InvariantI}.

\begin{proposition}[Attainment and monotonicity]
\label{prop:attainment-monotonicity}
Let $G=(V,E)$ be finite and connected. If $\DN_{\LLY}^{k}(G)<\infty$, then
the infimum is attained. If $0\le k_1\le k_2$, then
\begin{equation}
\label{eq:k-monotonicity-general}
\DN_{\LLY}^{k_1}(G)\le\DN_{\LLY}^{k_2}(G).
\end{equation}
\end{proposition}

\begin{proof}
The case $G=K_1$ is immediate. Assume $E(G)\ne\varnothing$. Choose a
$k$-admissible weight $w_0$ of finite distortion and normalize it so that
$\min_e w_{0,e}=1$. Put $M=\Dist(w_0)$. It suffices to minimize
over the nonempty set
\[
\mathcal A_M:=\{w\in[1,M]^E:\min_e w_e=1,\ 
\kappa_{\LLY}^w(e)\ge k\text{ for every }e\in E\}.
\]
By continuity of
$w\mapsto(\kappa_{\LLY}^w(e))_{e\in E(G)}$, $\mathcal A_M$ is closed and hence
compact. On $\mathcal A_M$, $\Dist(w)=\max_{e\in E} w_e$, a continuous
function, so the infimum is attained. The monotonicity follows because
every $k_2$-admissible weight is $k_1$-admissible.
\end{proof}

\subsection{Results used from Lin--Liu and the preceding paper}
\label{subsec:known-high-girth}

For a finite connected graph of girth at least $6$, Lin and Liu prove
\cite[Proposition~3.1]{LinLiu2026}
\begin{equation}
\label{eq:total-curvature}
\sum_{e\in E(G)}\kappa_{\LLY}^w(e)
=2(|V(G)|-|E(G)|).
\end{equation}
For $\Omega\subseteq V(G)$, write
\[
E(\Omega):=\{uv\in E(G):u,v\in\Omega\}.
\]
They also show \cite[Theorem~3.1]{LinLiu2026} that a positive
constant-curvature weight exists exactly when
\begin{equation}
\label{eq:LinLiu-density}
\max_{\varnothing\ne\Omega\subsetneq V(G)}
\frac{|E(\Omega)|}{|\Omega|}
<\frac{|E(G)|}{|V(G)|}.
\end{equation}
In that case the constant equals
\begin{equation}
\label{eq:average-curvature}
\overline\kappa
=2\left(\frac{|V(G)|}{|E(G)|}-1\right).
\end{equation}
In particular, every finite tree with $m$ edges satisfies the density
condition and hence admits a positive weight of constant curvature $2/m$.
Indeed, for every proper nonempty $\Omega\subsetneq V(T)$,
\[
\frac{|E(\Omega)|}{|\Omega|}
\le 1-\frac1{|\Omega|}
<1-\frac1{m+1}
=\frac{m}{m+1}=\frac{|E(T)|}{|V(T)|}.
\]

We shall also use three results from our preceding paper
\cite{Xia2026InvariantI}. First, its forest obstruction states that if
$\DN_{\LLY}(G)<\infty$, then either $G=C_n$ for some $n\ge6$ or
$\Skel(G)$ is a forest. Second, its tree-like subgraph comparison gives,
for every connected subgraph $H\subseteq G$ with
$E(H)\subseteq\Etl(G)$,
\begin{equation}
\label{eq:partI-tree-like-comparison}
\DN_{\LLY}(H)\le\DN_{\LLY}(G).
\end{equation}
Third, for every finite tree $T$, its branch-vertex bound gives
\begin{equation}
\label{eq:partI-branch-bound}
|B(T)|\le\left\lceil\DN_{\LLY}(T)\right\rceil.
\end{equation}

\subsection{The \texorpdfstring{$2$}{2}-core and boundary loads}

\begin{definition}\label{def:2-core}
For a finite graph $G$, repeatedly remove vertices whose current degree
is $0$ or $1$. The subgraph that remains (possibly empty) is the
\emph{$2$-core} $\core(G)$; equivalently, it is the unique maximal
subgraph of $G$ having minimum degree at least $2$. We call the vertices in $\core(G)$ the \emph{core vertices}.
\end{definition}

Let $G$ be finite and connected with nonempty $2$-core
$K:=\core(G)$. Every component of $G-V(K)$ is a tree attached to a
unique core vertex by one edge; we call these components the
\emph{hanging components} of $G$. Recall that
\[
A_G=\{x\in V(K):N_G(x)\setminus V(K)\ne\varnothing\},
\]
and for every $x\in A_G$, adjoin a new pendant vertex $x^\partial$ adjacent
to $x$; the resulting graph is denoted by $\widehat K_G$.

The following aggregation identity was proved in our preceding paper
\cite{Xia2026InvariantI} and will be used below.

\begin{lemma}\cite[Lemma~3.4]{Xia2026InvariantI}
\label{lem:boundary-aggregation}
Let $G$ be finite and connected with nonempty $2$-core
$K:=\core(G)$, and let $A_G$ and $\widehat K_G$ be as above. Let $w$
be a positive edge weight on $G$. For $x\in A_G$, put
\[
b_x:=\sum_{u\in N_G(x)\setminus V(K)}w_{xu},
\]
and give $\widehat K_G$ the weight
\[
\widehat w_e=w_e\quad(e\in E(K)),
\qquad
\widehat w_{xx^\partial}=b_x.
\]
Then every core edge $x\sim y$ satisfies
\begin{equation}
\label{eq:boundary-aggregation}
\kappa_{\LLY}^{w,G}(x,y)
=
\kappa_{\LLY}^{\widehat w,\widehat K_G}(x,y).
\end{equation}
\end{lemma}

\section{Positive curvature and core reduction}
\label{sec:positive-core}

\subsection{Bonnet--Myers and infinite graphs}

The discrete Bonnet--Myers argument in this setting follows the standard proof; see
\cite[Proposition~23]{Ollivier2009} and
\cite[Theorem~4.1]{LinLuYau2011}.
We include the proof for completeness, noting that the transportation
cost is given by the combinatorial graph distance.

\begin{proof}[Proof of \cref{prop:BM}]
Let $x=x_0,\ldots,x_n=y$ be a shortest path from $x$ to $y$, and fix
$0<\varepsilon<k$. Since, by \eqref{eq:kappa_LLY},
\[
\kappa_{\LLY}^w(x_{i-1},x_i)
=
\lim_{\alpha\to1^-}
\frac{\kappa_\alpha^w(x_{i-1},x_i)}{1-\alpha}
\ge k
\]
for every $i=1,\ldots,n$, for $\alpha<1$ sufficiently close to $1$ we have,
simultaneously for all $i$,
\[
\frac{\kappa_\alpha^w(x_{i-1},x_i)}{1-\alpha}
\ge k-\varepsilon.
\]
Combining with \eqref{eq:kappa_alpha_W}, we have
\[
W_1(\mu_{x_{i-1}}^{\alpha,w},\mu_{x_i}^{\alpha,w})
\le
1-(1-\alpha)(k-\varepsilon).
\]
Writing $\delta_z$ for the Dirac measure at $z$, we have
$W_1(\delta_z,\mu_z^{\alpha,w})=1-\alpha$. Hence, by the triangle inequality,
\[
\begin{aligned}
d(x,y)=W_1(\delta_x,\delta_y)\le
W_1(\delta_x,\mu_x^{\alpha,w})
+\sum_{i=1}^n
W_1(\mu_{x_{i-1}}^{\alpha,w},\mu_{x_i}^{\alpha,w})
+W_1(\mu_y^{\alpha,w},\delta_y).
\end{aligned}
\]
Since $d(x,y)=n$, the preceding estimates give
\[
n\le2(1-\alpha)
+n\bigl(1-(1-\alpha)(k-\varepsilon)\bigr).
\]
Thus $n(k-\varepsilon)\le2$. Letting $\varepsilon\downarrow0$ gives
$d(x,y)\le2/k$.
\end{proof}

\begin{corollary}\label{cor:no-infinite-positive-k}
If $G$ is connected, locally finite, and infinite, then
\[
\DN_{\LLY}^{k}(G)=\infty
\qquad(k>0).
\]
\end{corollary}

\begin{proof}
A finite-distortion $k$-admissible weight would give finite diameter by
\cref{prop:BM}; a locally finite connected graph of finite diameter is
finite.
\end{proof}

\subsection{Exact rooted-tree response}
\label{subsec:rooted-response}

\begin{theorem}[Rooted-tree response at level $k$]
\label{thm:rooted-response}
Let $k>0$, and let $G$ be a connected locally finite graph. Let $C$ be
a finite tree with $t:=|V(C)|\ge1$, rooted at $r$, and suppose that
$C$ is attached to $G-V(C)$ only through the boundary edge $e=xr$,
where $x\notin V(C)$. If a positive weight on $G$ makes $e$ and every
edge of $C$ have curvature at least $k$, then
\begin{equation}
\label{eq:rooted-share-lower}
p_x^w(e)=\frac{w_e}{d_x^w}\ge\frac{kt}{2}.
\end{equation}
If $A_x=d_x^w-w_e$ and $kt<2$, this is equivalent to
\begin{equation}
\label{eq:rooted-ratio-lower}
w_e\ge\frac{kt}{2-kt}A_x.
\end{equation}
If $kt=2$, it forces $A_x=0$.

Conversely, if the positive weights on all ambient edges incident with
$x$ are fixed and \eqref{eq:rooted-share-lower} holds, then the edges
of $C$ can be assigned positive weights so that every internal edge of
$C$ has curvature exactly $k$ and $e$ has curvature at least $k$.
Equality in \eqref{eq:rooted-share-lower} makes the boundary edge curvature
also equal to $k$.
\end{theorem}

\begin{proof}
Put $c_k=1+k/2$. Since the boundary edge $e$ and every edge of $C$ are tree-like, \eqref{eq:share-curvature} applies. Summing it over the $t-1$ internal
edges of $C$ gives
\[
\sum_{uv\in E(C)}
\bigl(p_u^w(uv)+p_v^w(uv)\bigr)
=t-p_r^w(e)\ge(t-1)c_k,
\qquad
p_r^w(e)\le1-\frac{k(t-1)}2.
\]
Applying \eqref{eq:share-curvature} to $e=xr$ yields
\[
p_x^w(e)\ge c_k-p_r^w(e)\ge\frac{kt}{2},
\]
proving necessity; \eqref{eq:rooted-ratio-lower} is algebraically
equivalent.

For sufficiency, \eqref{eq:rooted-share-lower} implies $kt\le2$.
Root $C$ at $r$. For $v\in V(C)$ let $t_v$ be the number of vertices
in the descendant subtree rooted at $v$ and set
\[
\sigma_v:=1-\frac{k(t_v-1)}2.
\]
Since $t_v\le t$ and $kt\le2$, one has $\sigma_v>0$.
For each vertex $u\in V(C)$, let $e_u$ be its parent edge, with $e_r=e$. Set $[u]:=\{v\in V(C): v\ \text{is a child of}\ u\}$, and define recursively
\begin{equation}
\label{eq:direct-rooted-weight}
w_{uv}:=\frac{k t_v}{2\sigma_u}\,w_{e_u}.
\end{equation}
Since the child subtree sizes sum to $t_u-1$,
\[
d_u^w=w_{e_u}+\sum_{v\in [u]}w_{uv}=\frac{w_{e_u}}{\sigma_u},
\qquad
p_u^w(e_u)=\sigma_u,
\qquad
p_u^w(uv)=\frac{k t_v}{2}.
\]
Hence for every internal parent--child edge $uv$,
\[
p_u^w(uv)+p_v^w(uv)
=\frac{k t_v}{2}+1-\frac{k(t_v-1)}2
=1+\frac{k}{2},
\]
so its curvature is $k$ by \eqref{eq:share-curvature}. At the root,
$p_r^w(e)=1-k(t-1)/2$; together with
$p_x^w(e)\ge kt/2$, this gives curvature at least $k$ on $e$ by \eqref{eq:share-curvature}, with
equality when the inequality in \eqref{eq:share-curvature} is sharp.
\end{proof}

Thus the lower bound
\[
p_x^w(e)\ge \frac{k|V(C)|}{2}
\]
for the boundary edge of a hanging component $C$ depends only on
$|V(C)|$, not on the structure of $C$.

\subsection{The \texorpdfstring{$k$}{k}-loaded \texorpdfstring{$2$}{2}-core}
\label{subsec:loaded-core}

Let $G$ be finite and connected and let $K=\core(G)\ne\varnothing$.
For $x\in V(K)$ let $\mathcal C_x$ be the hanging components attached
at $x$ and put
\begin{equation}
\label{eq:hanging-size}
h_G(x):=\sum_{C\in\mathcal C_x}|V(C)|.
\end{equation}
Thus $h_G(x)>0$ exactly for $x\in A_G$.

\begin{definition}[$k$-loaded core weight]
\label{def:k-loaded}
A positive weight $\widehat w$ on $\widehat K_G$ is
\emph{$k$-loaded admissible} if
\begin{align}
\kappa_{\LLY}^{\widehat w,\widehat K_G}(e)&\ge k
&& (e\in E(K)),\label{eq:kloaded-core}\\
\kappa_{\LLY}^{\widehat w,\widehat K_G}(xx^\partial)&\ge k h_G(x)
&& (x\in A_G).\label{eq:kloaded-boundary}
\end{align}
\end{definition}

Since $xx^\partial$ is pendant, \eqref{eq:kloaded-boundary} is
equivalent to
\begin{equation}
\label{eq:kloaded-share}
\frac{\widehat w_{xx^\partial}}{d_x^{\widehat w}}
\ge\frac{k h_G(x)}2.
\end{equation}
If $kh_G(x)<2$, it is also equivalent to
\begin{equation}
\label{eq:kloaded-ratio}
\widehat w_{xx^\partial}
\ge
\frac{k h_G(x)}{2-kh_G(x)}
\sum_{\substack{e\in E(K)\\ e\ni x}}\widehat w_e.
\end{equation}

\begin{proof}[Proof of \cref{thm:loaded-core-intro}]
For $k=0$, the assertion is exactly the $2$-core reduction of
\cite{Xia2026InvariantI}, so assume $k>0$.

Let $w$ be $k$-admissible on $G$, and aggregate the first-layer weights
at each $x\in A_G$ by setting
\[
b_x:=\sum_{u\in N_G(x)\setminus V(K)}w_{xu}.
\]
By \cref{lem:boundary-aggregation}, core
curvatures are unchanged. If $e_C$ is the boundary edge of a component
$C\in\mathcal C_x$, \cref{thm:rooted-response} gives
\[
\frac{w_{e_C}}{d_x^w}\ge\frac{k|V(C)|}{2}.
\]
Summing over $C\in\mathcal C_x$ gives
$b_x/d_x^w\ge kh_G(x)/2$. Set 
\[
\widehat w_e:=w_e\quad(e\in E(K)),
\qquad
\widehat w_{xx^\partial}:=b_x.
\] 
Since aggregation preserves $d_x^w$, the
resulting proxy weight $\widehat w$ is $k$-loaded admissible.

Conversely, let $\widehat w$ be $k$-loaded admissible and fix
$x\in A_G$. Write $\mathcal C_x=\{C_1,\ldots,C_q\}$,
$t_i=|V(C_i)|$, and $b_x=\widehat w_{xx^\partial}$. By
\eqref{eq:kloaded-share}, choose positive $b_i$ with
\[
\sum_i b_i=b_x,
\qquad
\frac{b_i}{d_x^{\widehat w}}\ge\frac{k t_i}{2}.
\]
Assign $b_i$ to the boundary edge of $C_i$ and apply the converse part
of \cref{thm:rooted-response} independently to each component, then we can reconstruct a positive weight $w_G$ on G such that every edge of
$G\setminus K$ has curvature exactly $k$ and the boundary edge of $C_i$ has curvature at least $k$. The
total first-layer load at every core vertex remains $b_x$, so
\cref{lem:boundary-aggregation} preserves all core curvatures. The
reconstructed weight $w_G$ on $G$ is therefore $k$-admissible.
\end{proof}

\begin{corollary}[Hanging-size obstruction]
\label{cor:hanging-size}
Let $G$ be finite and connected with nonempty $2$-core
$K:=\core(G)$. If $k>0$ and $\DN_{\LLY}^{k}(G)<\infty$, then
\begin{equation}
\label{eq:hanging-size-obstruction}
h_G(x)<\frac{2}{k}
\qquad(x\in A_G).
\end{equation}
\end{corollary}

\begin{proof}
Let $x\in A_G$. By \eqref{eq:kloaded-share},
\[
\frac{\widehat w_{xx^\partial}}{d_x^{\widehat w}}
\ge\frac{k h_G(x)}2.
\]
It is strictly less than $1$ because $x$ also carries
positive core-edge weight.
\end{proof}

\section{Finite trees: fixed points and the curvature profile}
\label{sec:finite-trees}

Throughout this section, $T$ is a finite tree with $E(T)\ne\varnothing$.
The convention $\DN_{\LLY}^{k}(K_1)=1$ remains in force.

\subsection{The restoration map}

\begin{proposition}[Sharp feasible interval]
\label{prop:tree-feasible-range}
If $m:=|E(T)|$, then
\begin{equation}
\label{eq:tree-feasible-range}
\DN_{\LLY}^{k}(T)<\infty
\quad\Longleftrightarrow\quad
0\le k\le\frac{2}{m}.
\end{equation}
\end{proposition}

\begin{proof}
If every edge has curvature at least $k$, then
\eqref{eq:total-curvature} gives $mk\le2$. Conversely, by the constant-curvature existence result of Lin--Liu \cite[Theorem~3.1]{LinLiu2026},
there is a positive weight $w$ on $T$ such that
\[
\kappa_{\LLY}^w(e)=\frac{\sum_{e\in E}\kappa_{\LLY}^w(e)}{m}=\frac{2(|V(T)|-|E(T)|)}{m}=\frac{2}{m}
\qquad(e\in E(T)).
\]
Hence this weight is $k$-admissible for every
$0\le k\le2/m$.
\end{proof}

For $0\le k<2$ and $a,b\ge0$, recall that
\begin{equation}
\label{eq:Psi}
\Psi_k(a,b)
=
\frac{k(a+b)+\sqrt{k^2(a-b)^2+16ab}}{2(2-k)}.
\end{equation}
Note that $\Psi_k$ is positively homogeneous of degree one:
\[
\Psi_k(\lambda a,\lambda b)=\lambda\Psi_k(a,b)
\qquad(\lambda>0).
\]

\begin{lemma}[Restoration function]
\label{lem:Psi-properties}
For $0\le k<2$, $\Psi_k$ is continuous, positively homogeneous of
degree one, and nondecreasing in each variable; for $k>0$ it is
strictly increasing in each variable.
\end{lemma}

\begin{proof}
The first two assertions are immediate. For $k=0$,
\[
\Psi_0(a,b)=\sqrt{ab},
\]
which is nondecreasing in each variable.

Assume $k>0$. For $(a,b)\ne(0,0)$, direct differentiation gives
\[
\frac{\partial \Psi_k}{\partial a}
=
\frac{1}{2(2-k)}
\left(
k+
\frac{k^2(a-b)+8b}
{\sqrt{k^2(a-b)^2+16ab}}
\right).
\]
If $a\ge b$, this is clearly positive. If $a<b$, then
\[
\sqrt{k^2(a-b)^2+16ab}\ge k(b-a),
\]
and hence
\[
k\sqrt{k^2(a-b)^2+16ab}
+k^2(a-b)+8b
\ge 8b>0.
\]
Thus
\[
\frac{\partial \Psi_k}{\partial a}>0.
\]
By symmetry,
\[
\frac{\partial \Psi_k}{\partial b}>0.
\]
At the origin, strict increase along the coordinate axes follows from
\[
\Psi_k(a,0)=\frac{k}{2-k}a,
\qquad
\Psi_k(0,b)=\frac{k}{2-k}b.
\]
Therefore $\Psi_k$ is strictly increasing in each variable for $k>0$.
\end{proof}

\begin{lemma}[$k$-restoration formula]
\label{lem:k-restoration}
Let $G$ be connected and locally finite, let $w$ be a positive edge
weight, and let $xy\in\Etl(G)$. Put
\[
A_x=d_x^w-w_{xy},
\qquad
A_y=d_y^w-w_{xy}.
\]
For $0\le k<2$,
\begin{equation}
\label{eq:k-restoration-equivalence}
\kappa_{\LLY}^w(x,y)\ge k
\quad\Longleftrightarrow\quad
w_{xy}\ge\Psi_k(A_x,A_y),
\end{equation}
with equality on the right exactly when the curvature equals $k$.
\end{lemma}

\begin{proof}
Substituting \eqref{eq:tree-like-curvature} and clearing denominators
gives
\[
(2-k)w_{xy}^2-k(A_x+A_y)w_{xy}-(2+k)A_xA_y\ge0.
\]
Its unique nonnegative root is \eqref{eq:Psi}.
\end{proof}

For $x=(x_e)_{e\in E(T)}\in[1,\infty)^{E(T)}$ and $e=u\sim v$, recall that
\begin{equation}
\label{eq:Auv-all-edges}
A_{u\mid v}(x)
=\sum_{\substack{f\in E(T)\\ f\ni u,\ f\ne e}}x_f.
\end{equation}

\begin{definition}[$k$-restoration map]
\label{def:k-restoration-map}
For $0\le k<2$, define
\[
F_{T,k}:[1,\infty)^{E(T)}\longrightarrow[1,\infty)^{E(T)}
\]
by
\begin{equation}
\label{eq:k-restoration-map}
(F_{T,k}(x))_{uv}
:=\max\left\{1,
\Psi_k(A_{u\mid v}(x),A_{v\mid u}(x))
\right\}.
\end{equation}
We call a vector $x\in [1,\infty)^{E(T)}$ a \emph{fixed point} of $F_{T,k}$ if 
\begin{equation}\label{eq:distortion-system}
    F_{T,k}(x)=x,
\end{equation}
and \emph{$k$-feasible} if $F_{T,k}(x)\le x$ coordinatewise. It is called \emph{normalized} if \[ \min_{e\in E(T)}x_e=1. \] 
\end{definition}

Viewing a normalized weight as a vector
$x\in [1,\infty)^{E(T)}$, \cref{lem:k-restoration} gives
\begin{equation}
\label{eq:k-feasible-super}
x\text{ is a $k$-admissible weight on $T$}
\quad\Longleftrightarrow\quad
x\text{ is a $k$-feasible vector in $[1,\infty)^{E(T)}$}.
\end{equation}
Moreover, since $\Psi_0(a,b)=\sqrt{ab}$,
\begin{equation}
\label{eq:F0-equals-F}
(F_{T,0}(x))_{uv}
=\max\left\{1,
\sqrt{A_{u\mid v}(x)A_{v\mid u}(x)}
\right\}.
\end{equation}
This is exactly the all-edge LLY-distortion map used in
\cite{Xia2026InvariantI}.

\begin{theorem}[Finite-tree fixed-point formula]
\label{thm:k-tree-fixed-point}
Let $m:=|E(T)|\ge2$ and $0\le k\le2/m$. Then $F_{T,k}$ has a unique
normalized fixed point $w_T^{\star,k}$, which is coordinatewise minimal
among normalized $k$-admissible weights, and
\begin{equation}
\label{eq:k-tree-DN-fixed}
\DN_{\LLY}^{k}(T)=\max_{e\in E(T)}w_{T,e}^{\star,k}.
\end{equation}
For $T=K_2$, the unit weight is the unique normalized weight and
$\DN_{\LLY}^{k}(K_2)=1$ for $0\le k\le2$.
\end{theorem}

\begin{proof}
By \cref{prop:tree-feasible-range}, choose a normalized
$k$-admissible vector $x^{(0)}$. Iterating
$x^{(n+1)}=F_{T,k}(x^{(n)})$ gives a coordinatewise decreasing sequence
bounded below by $\mathbf1$, because $F_{T,k}$ is order preserving.
Its limit is a fixed point by continuity. Since $x^{(0)}$ is normalized, some coordinate initially equals $1$;
that coordinate remains $1$ throughout the iteration. Hence the limit
is normalized.

For uniqueness, the case $k=0$ is exactly the finite-tree fixed-point
result of \cite{Xia2026InvariantI}, by \eqref{eq:F0-equals-F}. Let $k>0$ and let $x,y$ be normalized fixed
points. Set
\[
R:=\max_{e\in E(T)}\frac{x_e}{y_e}.
\]
Without loss of generality, let $R\ge1$. If $R>1$, choose $e=u\sim v$ attaining it. Then $x_e=Ry_e>1$, and
homogeneity and monotonicity of $\Psi_k$ give
\[
x_e
=
\Psi_k(A_{u\mid v}(x),A_{v\mid u}(x))
\le
R\Psi_k(A_{u\mid v}(y),A_{v\mid u}(y))
\le
Ry_e
=
x_e.
\]
Equality holds throughout. Strict monotonicity therefore forces
\[
A_{u\mid v}(x)=R A_{u\mid v}(y),
\qquad
A_{v\mid u}(x)=R A_{v\mid u}(y).
\]
The equality of the corresponding sums,
together with $x_f\le R y_f$ termwise, forces $x_f=R y_f$ on every
adjacent edge $f$. Repeating this along the connected edge set of the tree
gives $x=Ry$, contradicting normalization. Exchanging $x$ and $y$ gives $y\le x$, and hence $x=y$.

Starting the iteration from any normalized $k$-admissible vector $z$
produces the unique fixed point below $z$, proving coordinatewise
minimality. Since the fixed point is itself $k$-admissible,
\eqref{eq:k-tree-DN-fixed} follows by \cref{prop:attainment-monotonicity}. For $T=K_2$, the assertion is immediate by \cref{prop:tree-feasible-range}.
\end{proof}

\begin{definition}[$k$-canonical weight]
\label{def:canonical-k-weight}
For $0\le k\le2/|E(T)|$, the normalized fixed point in
\cref{thm:k-tree-fixed-point} is the \emph{$k$-canonical weight}
$w_T^{\star,k}$; for $T=K_2$ it is the unit weight. By \eqref{eq:F0-equals-F}, $w_T^{\star,0}$ is exactly the normalized
canonical weight constructed in \cite{Xia2026InvariantI}.
\end{definition}

\begin{theorem}[Uniqueness of the normalized $k$-optimal weight]
\label{thm:k-optimal-weight-uniqueness}
For $0\le k\le2/|E(T)|$, $w_T^{\star,k}$ is the unique normalized
$k$-optimal weight. Equivalently, every $k$-optimal
weight is $c\,w_T^{\star,k}$ for some $c>0$.
\end{theorem}

\begin{proof}
At $k=0$ this is the uniqueness assertion in the finite-tree
fixed-point theorem of \cite{Xia2026InvariantI}. The case $T=K_2$ is immediate,
so assume $k>0$ and $|E(T)|\ge2$. Let $w$ be normalized and $k$-optimal.
By coordinatewise minimality,
\[
w\ge w_T^{\star,k}.
\]
Put $D=\DN_{\LLY}^{k}(T)=\max_e w_{T,e}^{\star,k}$. If $D=1$, both
weights are the unit weight.

Assume $D>1$ and suppose that
\[
w_f>w_{T,f}^{\star,k}
\]
for some $f\in E(T)$. Let $e=uv$ be adjacent to $f$, with
$f\ni u$, and assume $w_{T,e}^{\star,k}>1$. Since
$w\ge w_T^{\star,k}$,
\[
A_{u\mid v}(w)>A_{u\mid v}(w_T^{\star,k}),
\qquad
A_{v\mid u}(w)\ge A_{v\mid u}(w_T^{\star,k}).
\]
By $k$-admissibility of $w$ and strict monotonicity of $\Psi_k$,
\[
w_e\ge
\Psi_k\bigl(A_{u\mid v}(w),A_{v\mid u}(w)\bigr)>
\Psi_k\bigl(
A_{u\mid v}(w_T^{\star,k}),
A_{v\mid u}(w_T^{\star,k})
\bigr)=
w_{T,e}^{\star,k}.
\]
Thus strict inequality propagates from $f$ to every adjacent edge whose
canonical weight is greater than $1$.

For every non-pendant edge $e=uv$,
\[
A_{u\mid v}(w_T^{\star,k})\ge1,
\qquad
A_{v\mid u}(w_T^{\star,k})\ge1,
\]
and hence
\[
w_{T,e}^{\star,k}
\ge\Psi_k(1,1)>1.
\]
Choose $e_0$ such that
\[
w_{T,e_0}^{\star,k}=D.
\]
Let
\[
f=f_0,f_1,\ldots,f_r=e_0
\]
be a shortest sequence of pairwise adjacent edges. Every
$f_j$, $1\le j<r$, is non-pendant, and $w_{T,e_0}^{\star,k}=D>1$.
Applying the preceding comparison successively gives
\[
w_{f_j}>w_{T,f_j}^{\star,k}
\qquad(j=1,\ldots,r),
\]
and therefore
\[
w_{e_0}>w_{T,e_0}^{\star,k}=D,
\]
contrary to $\max_{e\in E(T)}w_e=D$. Hence
\[
w=w_T^{\star,k}.
\]
The unnormalized statement follows from scale invariance.
\end{proof}

\begin{remark}
Lin and Liu \cite[Proposition~3.2]{LinLiu2026} show that, on a finite
high-girth graph, prescribing the curvature of every edge determines the
weight up to scaling. In \cref{thm:k-optimal-weight-uniqueness}, the
individual edge curvatures are not prescribed; only the lower bound
\[
\kappa_{\LLY}^w(e)\ge k
\qquad(e\in E(T))
\]
is imposed. Nevertheless, minimizing the distortion determines a unique
normalized weight. Moreover, this uniqueness phenomenon is special to
trees; see \cref{ex:nonuniqueness_phenomenon}.
\end{remark}

\subsection{Complementarity and profile regularity}
\label{subsec:profile-regularity}

\begin{corollary}[Complementarity of $w_T^{*,k}$]
\label{cor:complementarity}
For $0\le k\le2/|E(T)|$,
\begin{equation}
\label{eq:complementarity}
(w_{T,e}^{\star,k}-1)
\bigl(\kappa_{\LLY}^{w_T^{\star,k}}(e)-k\bigr)=0
\qquad(e\in E(T)).
\end{equation}
\end{corollary}

\begin{proof}
If $w_{T,e}^{\star,k}>1$, the second term in
\eqref{eq:k-restoration-map} is active, so
\cref{lem:k-restoration} gives curvature exactly $k$. If the curvature
is strictly larger than $k$, the contrapositive gives $w_{T,e}^{\star,k}=1$.
\end{proof}

Define
\begin{equation}
\label{eq:tree-profile-def}
\mathcal D_T(k):=\DN_{\LLY}^{k}(T),
\end{equation}
as a map
$k\mapsto \DN_{\LLY}^{k}(T)$.

\begin{theorem}[Profile regularity]
\label{thm:profile-regularity}
Let $m=|E(T)|$. For every edge $e$, the map
$k\mapsto w_{T,e}^{\star,k}$ is nondecreasing and continuous on
$[0,2/m]$. Consequently, $\mathcal D_T$ is nondecreasing and continuous
there. At $k=2/m$,
\begin{equation}
\label{eq:profile-endpoint}
\kappa_{\LLY}^{w_T^{\star,2/m}}(e)=\frac{2}{m}
\qquad(e\in E(T)).
\end{equation}
If $k<2/m$, at least one edge has curvature strictly larger than $k$.
\end{theorem}

\begin{proof}
The case $K_2$ is immediate. Let $m\ge2$ and
$0\le k_1\le k_2\le2/m$. The $k_2$-canonical weight is
$k_1$-admissible, so minimality gives
\begin{equation}
\label{eq:canonical-monotonicity}
w_T^{\star,k_1}\le w_T^{\star,k_2}
\end{equation}
coordinatewise.

If $k_j\to k$, then by \eqref{eq:canonical-monotonicity},
\[
1\le w_{T,e}^{\star,k_j}\le w_{T,e}^{\star,2/m}
\qquad(e\in E(T)).
\]
Hence the sequence $\{w_T^{\star,k_j}\}$ lies in the compact set
\[
\prod_{e\in E(T)}
\left[1,w_{T,e}^{\star,2/m}\right].
\]
Let a convergent subsequence satisfy
\[
w_T^{\star,k_{j_\ell}}\longrightarrow w.
\]
Since
\[
F_{T,k_{j_\ell}}
\bigl(w_T^{\star,k_{j_\ell}}\bigr)
=
w_T^{\star,k_{j_\ell}}
\]
for every $\ell$, and since
$(k,x)\mapsto F_{T,k}(x)$ is continuous, passing to the limit gives
\[
F_{T,k}(w)=w.
\]
Thus $w$ is a fixed point of $F_{T,k}$. By uniqueness of the fixed point,
\[
w=w_T^{\star,k}.
\]
Thus every convergent subsequence of
$\{w_T^{\star,k_j}\}$ has the same limit $w_T^{\star,k}$.
Since the sequence lies in a compact set, it follows that
\[
w_T^{\star,k_j}\longrightarrow w_T^{\star,k}.
\]
Hence the map
\[
k\longmapsto w_{T}^{\star,k}
\]
is continuous on $[0,2/m]$.
In particular, the map
\[
k\longmapsto w_{T,e}^{\star,k}
\]
is continuous on $[0,2/m]$ for every $e\in E(T)$. Since $E(T)$ is finite and
\begin{equation}\label{eq:D_T_k}
    \mathcal D_T(k)=\max_{e\in E(T)}w_{T,e}^{\star,k},
\end{equation}
continuity of each $k\mapsto w_{T,e}^{\star,k}$ implies continuity of
$\mathcal D_T$. Since each map $k\longmapsto w_{T,e}^{\star,k}$ is nondecreasing, the function $\mathcal D_T$ is nondecreasing as well by \eqref{eq:D_T_k}.

At $k=2/m$, all curvatures are at least $2/m$ and their sum is $2$ by
\eqref{eq:total-curvature}, so all are equal. If $k<2/m$ and all
curvatures equaled $k$, the same identity would give $2=mk<2$, a
contradiction.
\end{proof}

\begin{proof}[Proof of \cref{thm:tree-profile-intro}]
The sharp feasible interval is \cref{prop:tree-feasible-range}; the case
$K_2$ is included in \cref{thm:k-tree-fixed-point}. For $|E(T)|\ge2$,
the fixed-point formula and coordinatewise minimality are
\cref{thm:k-tree-fixed-point}, uniqueness of the normalized $k$-optimum is
\cref{thm:k-optimal-weight-uniqueness}, complementarity is
\cref{cor:complementarity}, and monotonicity, continuity, and the endpoint
constant-curvature statement are \cref{thm:profile-regularity}.
\end{proof}

\subsection{Examples}

\begin{example}[Stars]
\label{ex:star-profile}
For $T=K_{1,m}$, the unit weight has curvature $2/m$ on every edge.
Thus
\[
\DN_{\LLY}^{k}(K_{1,m})=1
\qquad(0\le k\le2/m).
\]
\end{example}

\begin{example}[The path $P_4$]
\label{ex:P4-profile}
Let
\[
P_4:=v_1-v_2-v_3-v_4.
\]
By symmetry, the $k$-canonical weight has common pendant weight $a$ and
middle-edge weight $t$. The fixed-point equation \eqref{eq:distortion-system} for the middle edge gives
\[
t=\Psi_k(a,a)=a\frac{2+k}{2-k}\ge a.
\]
Normalization forces $a=1$, hence
\begin{equation}
\label{eq:P4-profile}
\DN_{\LLY}^{k}(P_4)=\frac{2+k}{2-k}
\qquad(0\le k\le2/3).
\end{equation}
By \cref{eq:tree-like-curvature}, the curvatures for pendants are
\[
\kappa_{\LLY}^{w_{P_4}^{\star,k}}(v_1v_2)
=
\kappa_{\LLY}^{w_{P_4}^{\star,k}}(v_3v_4)
=
1-\frac{k}{2},
\] strictly larger than $k$ for
$k<2/3$, while all three curvatures equal $2/3$ at the endpoint.
\end{example}

\begin{example}[A pendant $k$-canonical weight above one]
\label{ex:pendant-weight-greater-one}
Let $T$ be the double star $ST_{1,2}$, obtained by attaching one leaf
to one endpoint of the center edge and two leaves to the other, and take
the endpoint $k=1/2$. By symmetry, let $a$
be the single pendant weight at one center, $b$ the common weight of
the two pendant edges at the other center, and $t$ the center-edge
weight. By \cref{thm:tree-profile-intro}, 
\[
\kappa_{\LLY}^{w_T^{\star,1/2}}(e)=\frac{1}{2}
\qquad(e\in E(T)).
\]
Then \eqref{eq:tree-like-curvature} gives
\[
\frac{2a}{a+t}=\frac12,
\qquad
\frac{2b}{2b+t}=\frac12,
\]
so $t=3a=2b$. After normalization,
\[
a=1,\qquad t=3,\qquad b=\frac32.
\]
Thus positive $k$ can force pendant $k$-canonical weights strictly above
$1$, a phenomenon absent at $k=0$; see \cite[Theorem~1.3]{Xia2026InvariantI}.
\end{example}

\begin{example}[Nonuniqueness on a non-tree graph]
\label{ex:nonuniqueness_phenomenon}
Let $G$ be the graph with
\[
V(G)=\{x,y,a,b,c,d\},
\qquad
E(G)=\{xy,xa,xb,ab,yc,yd\}.
\]
The same graph was used in \cite{Xia2026InvariantI} to show
nonuniqueness of normalized optimal weights at $k=0$.

Fix $0<k<1/3$ and put
\[
t_k:=\frac{2(2+k)}{2-k}.
\]
For
\[
1\le s\le
\min\left\{
t_k,\frac{2}{3k}-1
\right\},
\]
define a weight $w^{(s)}$ by
\[
w^{(s)}_{xy}=t_k,
\qquad
w^{(s)}_{xa}=w^{(s)}_{xb}
=w^{(s)}_{yc}=w^{(s)}_{yd}=1,
\qquad
w^{(s)}_{ab}=s.
\]
We use the same local curvature computation as in the nonuniqueness example of \cite[Remark~5.6]{Xia2026InvariantI} for $\kappa_{\LLY}^{w^{(s)}}(a,b)$ and $\kappa_{\LLY}^{w^{(s)}}(x,a)$:
\[
\kappa_{\LLY}^{w^{(s)}}(a,b)
=
2-\frac{1}{s+1}>k.
\]
\[
\kappa_{\LLY}^{w^{(s)}}(x,a)
=
\kappa_{\LLY}^{w^{(s)}}(x,b)
=
\frac{1}{s+1}-\frac{k}{2}\ge k.
\]
A direct computation by \cref{eq:tree-like-curvature} gives
\[
\kappa_{\LLY}^{w^{(s)}}(x,y)=k,
\]
\[
\kappa_{\LLY}^{w^{(s)}}(y,c)
=
\kappa_{\LLY}^{w^{(s)}}(y,d)
=
\frac{2-k}{4}>k.
\]
Hence $w^{(s)}$ is $k$-admissible and
\[
\Dist(w^{(s)})=t_k.
\]
Conversely, if $w$ is any normalized $k$-admissible weight on $G$, then
the tree-like edge $xy$ satisfies
\[
A_{x\mid y}(w)\ge2,
\qquad
A_{y\mid x}(w)\ge2.
\]
By \cref{lem:k-restoration},
\[
w_{xy}
\ge
\Psi_k(2,2)
=
\frac{2(2+k)}{2-k}
=t_k.
\]
Therefore
\[
\DN_{\LLY}^{k}(G)=t_k.
\]
Since
\[
\min\left\{
t_k,\frac{2}{3k}-1
\right\}>1
\qquad(0<k<1/3),
\]
the family $w^{(s)}$ contains a nontrivial continuum of distinct
normalized $k$-optimal weights.
\end{example}

\section{Tree-like skeletons and high-girth graphs}
\label{sec:skeleton}

\subsection{Comparison at level \texorpdfstring{$k$}{k}}

The tree-like subgraph comparison from \cite{Xia2026InvariantI}
extends to every prescribed lower bound.

\begin{theorem}[Tree-like subgraph comparison]
\label{thm:k-tree-like-comparison}
Let $G$ be connected and locally finite, let $k\ge0$, and let
$H\subseteq G$ be connected with $E(H)\subseteq\Etl(G)$. Then
\begin{equation}
\label{eq:k-tree-like-comparison}
\DN_{\LLY}^{k}(H)\le\DN_{\LLY}^{k}(G).
\end{equation}
\end{theorem}

\begin{proof}
If $E(H)=\varnothing$, then $H=K_1$ and the claim follows from
$\DN_{\LLY}^k(K_1)=1\le\DN_{\LLY}^k(G)$. If
$\DN_{\LLY}^k(G)=\infty$, there is nothing to prove. Otherwise, restrict any finite-distortion $k$-admissible weight $w$ on
$G$ to $H$. For any edge $xy\in E(H)$,
\[
d_{x,H}^{w|_H}\le d_{x,G}^w,
\qquad
d_{y,H}^{w|_H}\le d_{y,G}^w,
\]
so \eqref{eq:tree-like-curvature} shows that its curvature cannot
decrease. Restriction also cannot increase distortion. Indeed, $\Dist(w|_H)\le\Dist(w)$. Taking infima
gives the result.
\end{proof}

\begin{corollary}[Degree bound along a tree-like edge]
\label{cor:degree-response}
Let $G$ be finite and connected, $0\le k<2$, and
$D:=\DN_{\LLY}^{k}(G)<\infty$. Then every edge $xy\in\Etl(G)$ satisfies
\begin{equation}
\label{eq:degree-response}
D\ge\Psi_k(d_x-1,d_y-1).
\end{equation}
\end{corollary}

\begin{proof}
For a normalized $k$-optimal weight $w$, $1\le w_e\le D$. Hence on a tree-like
edge $xy$, we have
\[
A_x\ge d_x-1\ \text{and}\ A_y\ge d_y-1.
\] 
By \cref{lem:k-restoration} and the monotonicity of $\Psi_k$,
\[
D\ge w_{xy}
\ge \Psi_k(A_x,A_y)
\ge \Psi_k(d_x-1,d_y-1).
\]
This proves \eqref{eq:degree-response}.
\end{proof}

\subsection{Quantitative skeleton structure}

\begin{proof}[Proof of \cref{thm:quantitative-skeleton-intro}]
By \cref{cor:no-infinite-positive-k}, $G$ is finite. Note that $D=\DN_{\LLY}^k(G)$. \cref{prop:attainment-monotonicity} gives,
\[
\DN_{\LLY}(G)
=
\DN_{\LLY}^{0}(G)
\le
\DN_{\LLY}^{k}(G)
=
D
<\infty.
\]
The forest obstruction from \cite[Theorem~1.1]{Xia2026InvariantI} says that either
$\Skel(G)$ is a forest or
$G=C_n$ for some $n\ge6$. The cycle case is impossible for $k>0$ by
\eqref{eq:total-curvature}, whose left-hand side would be at least
$nk>0$ but equals $0$. Thus $\Skel(G)$ is a forest.

Let $T$ be a nontrivial component. By
\cref{thm:k-tree-like-comparison}, $\DN_{\LLY}^{k}(T)\le D$. \cref{thm:tree-profile-intro} gives
\[
|E(T)|\le\left\lfloor\frac2k\right\rfloor.
\]
Finally, by \cref{prop:attainment-monotonicity},
\[
\DN_{\LLY}(T)\le\DN_{\LLY}^{k}(T)\le D,
\]
so the branch-vertex bound from \cite[Theorem~1.5]{Xia2026InvariantI} yields
$|B(T)|\le\lceil D\rceil$.
\end{proof}

\subsection{Complete high-girth classification}

\begin{proof}[Proof of \cref{cor:girth6-intro}]
If $w$ is $k$-admissible, \eqref{eq:total-curvature} gives
\[
k|E(G)|\le2(|V(G)|-|E(G)|).
\]
Since $k>0$, the right-hand side is positive, so $|E(G)|<|V(G)|$.
Connectedness gives $|E(G)|\ge|V(G)|-1$, hence
$|E(G)|=|V(G)|-1$ and $G$ is a tree. Then \cref{thm:tree-profile-intro} gives
$k\le2/|E(G)|$. Conversely, the assertion follows directly from
\cref{thm:tree-profile-intro}.
\end{proof}

\begin{remark}[Complete locally finite high-girth picture]
Let $G$ be connected and locally finite, with $E(G)\ne\varnothing$ and
girth at least $6$. At $k=0$, \cite[Corollary~6.6, 6.7]{Xia2026InvariantI} gives finite
distortion exactly for finite trees, cycles $C_n$ with $n\ge6$, the
double ray, and one-ended trees obtained from a finite tree by attaching
a ray whose vertices away from the attachment point all have degree
$2$. For every $k>0$, finite $k$-distortion occurs exactly
for finite trees satisfying $k\le2/|E(G)|$.
\end{remark}

\section*{Acknowledgements}

The author used ChatGPT (OpenAI) as a research-assistance tool for
brainstorming, preliminary searches, and suggestions concerning exposition
and LaTeX organization. The author independently checked the mathematical
statements and references.

\end{document}